\documentclass[onefignum,onetabnum]{siamart171218}

\usepackage{tikz, tikz-cd}
\usepackage[utf8]{inputenc}
\usepackage{bbm}
\usepackage{enumitem}

\newcommand{\R}{\mathbb{R}}
\newcommand{\Z}{\mathbb{Z}}

\renewcommand{\Re}{\operatorname{Re}}

\renewcommand{\phi}{\varphi}
\renewcommand{\epsilon}{\varepsilon}

\newcommand{\im}{\mathrm{im \;}}

\newcommand{\Ac}{\mathcal{A}}

\newcommand{\Uc}{\mathcal{U}}

\newcommand{\Gc}{\mathcal{G}}
\newcommand{\Ic}{\mathcal{I}}

\DeclareMathOperator{\dom}{dom}

\newcommand{\cut}[1]{}

\usepackage{lipsum}
\usepackage{amsfonts}
\usepackage{graphicx}
\usepackage{epstopdf}
\usepackage{algorithmic}
\ifpdf
  \DeclareGraphicsExtensions{.eps,.pdf,.png,.jpg}
\else
  \DeclareGraphicsExtensions{.eps}
\fi

\newsiamremark{remark}{Remark}
\newsiamremark{hypothesis}{Hypothesis}
\crefname{hypothesis}{Hypothesis}{Hypotheses}
\newsiamthm{claim}{Claim}

\headers{The Hodge-Laplacian on the \v{C}ech-de Rham complex}{Wietse M. Boon, Daniel F. Holmen, Jan M. Nordbotten and Jon E. Vatne}

\title{The Hodge-Laplacian on the \v{C}ech-de Rham complex governs coupled problems}

\author{Wietse M. Boon\thanks{MOX Scientific Modeling and Computing, Department of Mathematics, Politecnico di Milano, Piazza Leonardo da Vinci 32, Milano, Italy}
\and Daniel F. Holmen\thanks{Center for Modeling of Coupled Subsurface Dynamics, Department of Mathematics, University of Bergen, Allégaten 41, Bergen, Norway, \email{daniel.holmen@uib.no}}
\and Jan M. Nordbotten\footnotemark[2] \and Jon E. Vatne\thanks{Department of Economics, BI Norwegian Business School, Kong Christian Frederiks plass 5, Bergen, Norway}, \\ \funding{WMB has received funding from the European Union's Horizon 2020 research and innovation programme under the Marie Skłodowska-Curie grant agreement No. 101031434. The work of JMN took place in the context of the
”Akademia” grant at the University of Bergen titled ”FracFlow” (funded by Equinor ASA).}}

\usepackage{amsopn}

\makeatletter
\newcommand*{\addFileDependency}[1]{
  \typeout{(#1)}
  \@addtofilelist{#1}
  \IfFileExists{#1}{}{\typeout{No file #1.}}
}
\makeatother

\newcommand*{\myexternaldocument}[1]{%
    \externaldocument{#1}%
    \addFileDependency{#1.tex}%
    \addFileDependency{#1.aux}%
}

\ifpdf
\hypersetup{
  pdftitle={Hodge-Laplacian on the \v{C}ech-de Rham complex},
  pdfauthor={Wietse M. Boon, Daniel F. Holmen, Jan M. Nordbotten and Jon E. Vatne}
}
\fi

\myexternaldocument{ex_supplement}

\begin{document}

\maketitle

\begin{abstract}
By endowing the \v{C}ech-de Rham complex with a Hilbert space structure, we obtain a Hilbert complex with sufficient properties to allow for well-posed Hodge-Laplace problems. We observe that these Hodge-Laplace equations govern a class of coupled problems arising from physical systems including elastically attached rods, multiple-porosity flow systems and 3D-1D coupled flow models. 
\end{abstract}


\section{Introduction} 
The \v{C}ech-de Rham complex, introduced in \cite{Weil1952}, is a double complex that has primarily been used as a tool in algebraic topology to compute cohomology groups and to prove results in homological algebra. We introduce weighted inner products on the \v{C}ech-de Rham complex, granting us a Hilbert complex which admits the compactness property. In turn, the codifferential and the corresponding Hodge-Laplace operator can be derived. Using the theory of evolutionary equations, we can also account for time-dependency in the the associated Hodge-Laplace problem. At an abstract level, we obtain well-posedness of these systems, an orthogonal decomposition, the Poincaré inequality, convergent mixed finite element approximations, finite-dimensional cohomology, as well as functional guaranteed a posteriori bounds.

The main motivation of this contribution is to share the observation that several models of physical systems of coupled domains have a \v{C}ech-de Rham Hodge-Laplacian structure. 
In this text, we detail the case of 1D elastic rods \cite{antman2005nonlinear},
multiple porosity models \cite{barenblatt1960basic,chen1989transient,straughan2017mathematical}, and mixed-dimensional coupling with high dimensionality gap \cite{koch2020modeling,koppl2018mathematical,kuchta2021analysis}. Additionally, a \v{C}ech-de Rham Hodge-Laplacian structure can be identified in several other applications such as in the vibration of elastically connected rods \cite{kelly2009free} and heat- and fluid transfer in layered materials 
\cite{yuan2022heat, bear2013dynamics}.
The developments presented herein mirror the observation that flow in fractured porous media is governed by a Hodge-Laplacian on a double complex \cite{mdG}.

\Cref{sub: CdR,sub: exterior derivative HL} present the mathematical background by introducing the \v{C}ech-de Rham complex and the Hodge-Laplace operator for the de Rham complex, respectively. \Cref{sec: HL on CdR} combines these two to form the main focus of this work, namely Hodge Laplace problems on the \v{C}ech-de Rham complex, and presents the main results from the perspective of Hilbert complexes. Finally, \Cref{sec: examples} presents three examples from physical applications that are governed by the Hodge-Laplace equations.

\subsection{The \v{C}ech-de Rham complex}\label{sub: CdR}\noindent
Informally speaking, the \v{C}ech-de Rham complex is constructed by combining two cochain complexes; the de Rham complex and the \v{C}ech complex. The de Rham complex consists of differential forms on a manifold $\Omega$, with the exterior derivative $d$ acting as a differential operator. The \v{C}ech complex, on the other hand, introduces an open cover $\Uc$ of $\Omega$ and employs an operator $\delta$, which takes differences of differential forms on the intersections of sets in $\Uc$. 

Herein, we focus on the aspects of this theory that are most relevant for our purposes and refer the reader to \cite{bott-tu} for a general exposition. Let $\Uc = \{U_i\}_{i \in \Ic}$ be an open cover, indexed by an ordered set of integers $\Ic$, of a bounded Lipschitz domain $\Omega \subset \mathbb{R}^n$. 
By convention, we use multi-indices (and sometimes multiple subscripts to emphasize the multi-index) to denote intersection $U_i = U_{i_0, ..., i_p} = U_{i_0} \cap ... \cap U_{i_p}$ and we let $\Ic^p$ denote the set of increasing multi-indices $i = (i_0, ..., i_p)$ with $U_{i} \ne \emptyset$. Note that $\Ic = \Ic^0$ corresponds to single indices and $\Ic^p$ are multi-indices of length $p+1$, and we will use the index $i$ for both singular indices $i \in \Ic$ and for multi-indices $i \in \Ic^p$.

We will throughout this text assume to be working with a finite \emph{good cover}, meaning that $\Ic$ is a finite set and all nonempty intersections $U_{i=(i_0, ..., i_p)}$ are diffeomorphic to $\R^n$. 
Let then $\Lambda^k(U_i)$ be the space of differential $k$-forms on a given open set $U_i$ for $i\in\Ic^p$. The spaces of differential forms together with the exterior derivative $d$ form a cochain complex, meaning that $d^2 = d \circ d = 0$. The complex $(\Lambda^\bullet(U_i), d)$ is called the \emph{de Rham complex}:
\begin{equation}
0 \to \Lambda^0(U_i) \xrightarrow[]{d} \Lambda^1(U_i) \xrightarrow[]{d} ... \xrightarrow[]{d} \Lambda^{n-1}(U_i) \xrightarrow[]{d} \Lambda^n(U_i) \to 0.
\end{equation}

For a fixed open cover $\Uc$, we write $\Ac^{p, \,q} := \bigoplus_{i \in \Ic^p} \Lambda^q(U_i)$ for the de Rham complex on the open cover and its intersections of degree $p$. We now proceed to define a differential operator for the degree of overlap. Consider first  each non-empty overlap $U_{i=(i_0, i_1)}$ with $i_0<i_1$ and $i \in \Ic^1$, and let the operator $\delta_{i}: \Lambda^q(U_{i_0}) \oplus \Lambda^q(U_{i_1}) \to \Lambda^q(U_i)$ compute the difference as $\delta_i(\alpha_{i_0}, \alpha_{i_1}) = (\alpha_{i_1} - \alpha_{i_0})|_{U_{i}}$. This is formally an abuse of notation for the more cumbersome $\alpha_{i_1}|_{U_{i}}- \alpha_{i_0}|_{U_{i}}$, but will not lead to any confusion. By considering all overlaps, we obtain the operator $\delta: \Ac^{0, \,q} \to \Ac^{1, \,q}$, such that $(\delta\alpha)_i = \delta_i(\alpha_{i_0}, \alpha_{i_1})$. 
This operator generalizes to intersections of higher degree  $\delta: \Ac^{p, q} \to \Ac^{p + 1, q}$ as follows:
\begin{align} \label{eq: difference operator}
(\delta\alpha)_i &= \sum_{j=0}^{p + 1} (-1)^j (\alpha_{i_0, ..., i_{j-1}, i_{j+1} ,..., i_{p + 1}})|_{U_i}, &
\forall i \in \Ic^{p + 1}, \alpha &\in \Ac^{p, q}.
\end{align}

As an example, consider an open cover $\Uc = \{U_0, U_1, U_2\}$ and let $\alpha = (\alpha_0, \alpha_1, \alpha_2) \in \Ac^{0,q}$, i.e. each component $\alpha_i$ is a differential form of degree $q$. Then 
\begin{equation}
\delta \alpha = \left((\alpha_1 - \alpha_0)|_{U_{0,1}}, (\alpha_2 - \alpha_0)|_{U_{0,2}}, (\alpha_2 - \alpha_1)|_{U_{1,2}} \right).  
\end{equation}
 If we again apply the operator $\delta$, we get 
\begin{equation}
\delta^2 \alpha = \left((\alpha_2 - \alpha_1) - (\alpha_2 - \alpha_0) + (\alpha_1 - \alpha_0) \right)|_{U_{0,1,2}} = 0.
\end{equation}
One  can readily show that in likeness to the exterior derivative $d$, the difference operator always satisfies $\delta^2 = 0$ and consequently it defines a cochain complex on $\Ac^{\bullet, q}$, called the \emph{\v{C}ech complex with values in $\Lambda^q$}. We refer to the double complex $(\Ac^{\bullet, \bullet}, (d, \delta))$  as the \emph{\v{C}ech-de Rham complex}. The double-graded complex can be turned into a single-graded \emph{total complex} by considering the anti-diagonals $\Ac^k := \bigoplus\limits_{p + q = k} \Ac^{p, q}$, which we also refer to as the \v{C}ech-de Rham complex. 

Since the difference operator is a finite alternating sum of restrictions, it commutes with the exterior derivative. The \emph{total differential} is given by
\begin{align}
    D^k: \Ac^k &\to \Ac^{k+1}, &
    D^k &= d + (-1)^k \delta.
\end{align}
Since $D^k$ acts on a direct sum of spaces of different degree (both in terms of forms and intersections), the degree of $d$ and $\delta$ is therefore necessarily determined from context based on the element of $\Ac^k$ they act on. Since the operators $d$ and $\delta$ commute, we get that 
\begin{equation}
D^{k+1} \circ D^k  = d^2 + (-1)^{k} d\delta + (-1)^{k+1} \delta d + (-1)^{2k+1} \delta^2 = 0.
\end{equation}
Since $D^k$ acts on each $\Ac^{p,q}$ with $p+q=k$, we frequently omit the indexing of the operators $d$ and $\delta$. Note that some authors choose to define the two differential operators in such a way that they are anti-commutative and then define the total differential as $D = d + \delta$. 

The diagram below shows the augmented \v{C}ech-de Rham complex. The leftmost column is the original de Rham complex on $\Omega$ and $r$ is the sum of the restriction map onto the sets $U_i \in \Uc$ with $i \in \Ic$. The remaining columns are the de Rham complexes on the intersections of sets of degree $p$. The rows are the \v{C}ech complexes with values in $\Lambda^q$, and the total complex consists of the anti-diagonals of the double complex $\Ac^{p,q}$: 
\begin{equation} \label{eq: double complex}
\begin{tikzcd}
\Lambda^n(\Omega) \arrow[rr, "r"]                &  & {\mathcal{A}^{0,n}} \arrow[r, "(-1)^n\delta"]           & {\mathcal{A}^{1,n}} \arrow[r]                           & ... \arrow[r]           & {\mathcal{A}^{p,n}}                \\
... \arrow[u, "d"] \arrow[rr, "r"]               &  & ... \arrow[r] \arrow[u, "d"]                            & ... \arrow[r] \arrow[u, "d"]                            & ... \arrow[r] \arrow[u] & ... \arrow[u, "d"]                 \\
\Lambda^1(\Omega) \arrow[rr, "r"] \arrow[u, "d"] &  & {\mathcal{A}^{0,1}} \arrow[r, "-\delta"] \arrow[u, "d"] & {\mathcal{A}^{1,1}} \arrow[r, "\delta"] \arrow[u, "d"]  & ... \arrow[u] \arrow[r] & {\mathcal{A}^{p,1}} \arrow[u, "d"] \\
\Lambda^0(\Omega) \arrow[rr, "r"] \arrow[u, "d"] &  & {\mathcal{A}^{0,0}} \arrow[r, "\delta"] \arrow[u, "d"]  & {\mathcal{A}^{1,0}} \arrow[r, "-\delta"] \arrow[u, "d"] & ... \arrow[u] \arrow[r] & {\mathcal{A}^{p,0}} \arrow[u, "d"]
\end{tikzcd}
\end{equation}

\subsection{The Hodge-Laplacian of the exterior derivative}
\label{sub: exterior derivative HL}\noindent
Let $\Omega \subset \mathbb{R}^n$ be a bounded Lipschitz domain as before, and let 
$\mathcal{N}$ denote the index set $\{1, ..., n\}$.  Similarly to the previous section, $\mathcal{N}^k$ denotes the set of increasing multi-indices of length $k$. For two differential $k$-forms $\alpha, \beta \in L^2 \Lambda^k(\Omega)$, we can write them as 
\begin{equation*}
\alpha = \sum_{j \in \mathcal{N}^k} \hat \alpha_j \; dx_{j_1} \wedge ... \wedge dx_{j_k}, \qquad  \beta = \sum_{j \in \mathcal{N}^k} \hat \beta_j \; dx_{j_1} \wedge ... \wedge dx_{j_k},
\end{equation*}
with $\hat \alpha_j, \hat \beta_j \in L^2(\Omega)$.
Using the volume form $\mathrm{vol}_\Omega = \bigwedge_{j = 1}^n dx_j$, the inner product on $\Lambda^k(\Omega)$ and its induced norm are given by
\begin{align} \label{eq: inner product}
\langle \alpha, \beta \rangle_{\Lambda^k(\Omega)} &:= \int_{\Omega} \sum_{j \in \mathcal{N}^k} \hat \alpha_j \hat \beta_j \; \mathrm{vol}_\Omega, &
\|\alpha\|_{\Lambda^k(\Omega)} &:= \sqrt{\langle \alpha, \alpha \rangle_{\Lambda^k(\Omega)}}.
\end{align}

To obtain a more general class of models, we will introduce spatially varying weights in the inner product.  We define the \emph{weighted inner product} by considering a collection $\{w_k\}_{k=0}^n$ of bijective, symmetric, bounded linear operators $w_k: L^2 \Lambda^k(\Omega) \to L^2 \Lambda^k(\Omega)$. We can describe the weighted inner product in terms of the unweighted inner product in the following way:
\begin{equation}
\langle \alpha, \beta \rangle_{\Lambda^{k}_w(\Omega)} = \langle w_k \alpha, \beta \rangle_{\Lambda^{k}(\Omega)}.
\end{equation}
The weights are identified with material parameters in \Cref{sec: examples}. Note that the unweighted inner product corresponds to each $w_k$ being equal to the identity operator. We will work with weighted inner products, with the understanding that we obtain the standard inner product by considering unit weights. 

The inner product gives rise to an adjoint operator $d^*$ of the exterior derivative, called the codifferential:
\begin{equation}
\langle \alpha, d^*_k \beta \rangle_{\Lambda^{k}_w(\Omega)} = \langle d^k \alpha, \beta \rangle_{\Lambda^{k+1}_w(\Omega)}.
\end{equation}
Some authors use $\delta$ to denote the codifferential. We reserve this letter for the difference operator from \eqref{eq: difference operator} associated to \v{C}ech complex. We sometimes emphasize that the codifferential arises from a weighted inner product by denoting it $d^{*,w}_{k}$.
\begin{proposition} \label{prop: weighted adjoint}
The weighted adjoint $d^{*,w}_{k}$ can be expressed in terms of the unweighted adjoint $d^{*,1}_{k}$ in the following way:
\begin{equation} \label{eq: weighted adjoint}
d^{*,w}_{k} = w_{k-1}^{-1} d_{k}^{*,1} w_{k}.
\end{equation}
Here the adjoints of the weights are defined in terms of the unweighted inner-products.
\end{proposition}
\begin{proof}
By the definition of the weighted inner product and adjoint operators, we have that
\begin{align*}
\langle d^{k-1}\alpha, \beta \rangle_{\Lambda^{k}_w(\Omega)} &= \langle w_k d^{k-1}\alpha, \beta \rangle_{\Lambda^{k}(\Omega)} \\
&= \langle w_k d^{k-1} (w_{k-1}^{-1} w_{k-1}) \alpha, \beta \rangle_{\Lambda^{k}(\Omega)} \\
&= \langle (w_k d^{k-1} (w_{k-1})^{-1}) w_{k-1} \alpha, \beta \rangle_{\Lambda^{k}(\Omega)} \\
&= \langle w_{k-1} \alpha, (w_k d^{k-1} w_{k-1}^{-1})^* \beta \rangle_{\Lambda^{k-1}(\Omega)} \\
&= \langle \alpha, w_{k-1}^{-1} d^{*,1}_{k} w_{k} \; \beta  \rangle_{\Lambda^{k-1}_w(\Omega)}.
\end{align*}
In the final equality, we used the symmetry of $w_{k}$ and $w_{k - 1}^{-1}$. 
The result \eqref{eq: weighted adjoint} now follows.
\end{proof}
The requirement for the weights to be bijective is evident from this proposition, as the expression for $d^{*,w}_{k}$ involves the inverse $w_{k-1}^{-1}$, which we require to be bounded. The differential together with the codifferential defines the (weighted) \emph{Hodge-Laplacian} of the exterior derivative,
\begin{equation}
\Delta_{d,w}^k = d^{k-1} d^{*,w}_{k} + d^{*,w}_{k+1} d^k.   
\end{equation}
The exterior derivative together with the associated codifferential and Hodge-Laplacian gives an orthogonal decomposition of differential forms which generalizes the Helmholtz decomposition for vector fields, called the \emph{Hodge decomposition}:
\begin{equation}
L^2 \Lambda^k_w(\Omega) = \im d^{k-1} \oplus \ker \Delta_{d,w}^k \oplus \im d^{*,w}_{k+1}.
\end{equation}
Furthermore, we can identify the kernel of the Hodge-Laplacian with the cohomology of the manifold: $\ker \Delta_d^k \cong \ker d^k / \im d^{k-1}$.
For contractible domains, the de Rham cohomology vanishes for $k>0$, and therefore we can write any differential form as the image of $d$ and $d^*$, i.e. we write $\alpha = d\beta + d^* \eta$.
\section{The Hodge-Laplacian on the \texorpdfstring{$L^2$}{L2} \v{C}ech-de Rham complex}
\label{sec: HL on CdR}\noindent
In this section, we construct inner products on the \v{C}ech-de Rham complex and show that it forms a Hilbert complex, i.e. a cochain complex where each space is a Hilbert space and the differential operators are closed, densely-defined unbounded linear operators. 

Moreover, we define the Hodge-Laplace operator for the \v{C}ech-de Rham complex and show that the complex satisfies the compactness property. Several theoretical results that hold for arbitrary closed Hilbert complexes are then summarized. In particular, we obtain the well-posedness of the Hodge-Laplace problem of the \v{C}ech-de Rham complex, meaning that mathematical models that fit into this framework will be well-posed.

\subsection{The \texorpdfstring{$L^2$}{L2} \v{C}ech-de Rham complex}\noindent
In this section, we combine \cref{sub: CdR} and \cref{sub: exterior derivative HL} to make the \v{C}ech-de Rham complex into a Hilbert complex. For a given open cover $\Uc$, we start by defining the Hilbert spaces containing square-integrable forms $L^2 \Ac^{p, q}$ and the spaces formed by the anti-diagonals of the complex $L^2 \Ac^k$ with $k = p+q$:
\begin{align}
L^2 \Ac^{p, q} &:= \bigoplus_{i \in \Ic^p} L^2 \Lambda^q(U_i), &
L^2 \Ac^k := \bigoplus_{p + q = k} L^2 \Ac^{p, q}.
\end{align}
In order to define a weighted inner product on the \v{C}ech-de Rham complex, we require a collection $\{w_k\}_{k \in \Z}$, where each $w_k$ consists of a set of $k+1$ weights:
\begin{equation}
 w_k = \{w^k_{i}: L^2 \Lambda^q(U_i) \to L^2 \Lambda^q(U_i): i \in \Ic^p, p \in \{0, ..., k\}, p+q = k\},
\end{equation}
where each $w^k_i$ is a bijective, symmetric, bounded linear operator as before. The weighted inner product is then defined as
\begin{equation}
\langle \alpha, \beta \rangle_{\Ac^k_w} = \langle w_k \alpha, \beta \rangle_{\Ac^k} = \sum_{p=0}^k \sum_{i \in \Ic^{p}}  \langle w^k_{i} \alpha_i, \beta_i \rangle_{\Lambda^{q}(U_i)}.    
\end{equation}
As an example, on an open cover $\Uc = \{U_0, U_1\}$, the inner products for $f = (f_0, f_1), g=(g_0, g_1) \in \Ac^0$, are given by
\begin{align}
\langle f, g \rangle_{\Ac^0_w} &= \langle w^0_{0} f_0, g_0 \rangle_{\Lambda^0(U_0)} + \langle w^0_{1} f_1, g_1 \rangle_{\Lambda^0(U_1)}.
\end{align} 
Similarly for $\alpha =((\alpha_0, \alpha_1), \alpha_{0,1}), \beta = ((\beta_0, \beta_1), \beta_{0,1}) \in \Ac^1 = \Ac^{0,1} \oplus \Ac^{1,0}$, we have the following expression for the inner product:
\begin{align}
\langle \alpha, \beta \rangle_{\Ac^1_w} &= \langle w^1_{0} \alpha_0, \beta_0 \rangle_{\Lambda^1(U_0)} + \langle w^1_{1} \alpha_1, \beta_1 \rangle_{\Lambda^1(U_1)} + \langle w^1_{0,1} \alpha_{0, 1}, \beta_{0, 1} \rangle_{\Lambda^0(U_{0, 1})}.
\end{align}
We define the \emph{$L^2$ \v{C}ech-de Rham complex} as $L^2 \Ac^k$ equipped with the total differential $D^k := d + (-1)^k \delta$. The total differential $D^k$ is not well-defined on the entire Hilbert space $L^2 \Ac^k$. However, we show that it is a closed densely-defined unbounded linear operator, see e.g. \cite[Chap. 3]{arnoldFEEC}, which is necessary to have a Hilbert complex.
\begin{proposition} \label{prop: CDUL}
$D^k: \dom D^k \subset L^2 \Ac^k \to L^2 \Ac^{k + 1}$ is a closed densely defined unbounded linear operator (CDUL for short).
\end{proposition}
\begin{proof}
There are three properties that need to be addressed: closed, densely defined and linear. Having an unbounded operator in not necessary, as a bounded operator can also be a CDUL operator. However, $D^k$ will be unbounded because the exterior derivative is unbounded. Linearity is obvious, since the sum of two linear operators is again linear.

We will show that $D^k$ is densely defined in $L^2 \Ac^k$. For a given $k$, the exterior derivative $d^q$ is a CDUL operator on $L^2 \Ac^{p, q}$ for each pair $(p,q)$ satisfying $p+q=k$. By direct summation, it follows that the sum $\oplus_q d^q$ is a CDUL on $L^2 \Ac^k$. Next, we note that $\oplus_q \dom d^q \subset L^2 \Ac^k$ and $\oplus_p \dom \delta^p = L^2 \Ac^k$ and so $\dom D^k = \oplus_p \dom \delta^p \cap \oplus_q \dom d^q  = \oplus_q \dom d^q$, which is dense in $L^2 \Ac^k$. 

It remains to show that $D^k$ is closed. Consider a sequence $\alpha_1, \alpha_2, ... \in \dom D^k$ such that $\alpha_m \to \alpha$ and $D\alpha_m \to \beta$, as $m \to \infty$, for some $\alpha \in L^2 \Ac^k$ and $\beta \in L^2 \Ac^{k + 1}$. By the continuity of $\delta$, we have that $d \alpha_m \to (\beta - (-1)^k \delta \alpha) \in L^2 \Ac^{k + 1}$. Since $d^k$ is closed, it follows that $\alpha \in \dom d^k = \dom D^k$ and $d \alpha = \beta - (-1)^k \delta \alpha$, which implies $D \alpha = \beta$.
\end{proof}

\subsection{The total Hodge-Laplace operator}
\label{sub: HL on CdR}\noindent
We continue by defining the total codifferential $D_k^{*,w}: \dom D_k^{*,w} \subset L^2 \Ac^k \to L^2 \Ac^{k - 1}$ through the adjoint relationship:
\begin{equation} \label{eq: total codiff}
    \langle D_k^{*,w}\alpha, \beta \rangle_{\Ac^{k - 1}_w} = \langle \alpha, D^{k - 1} \beta \rangle_{\Ac^k_w}
\end{equation}
\begin{remark} \label{remark: codifferential CDUL}
    The codifferential $D_k^{*,w}$ is also a CDUL on $L^2 \Ac^k$ by \cite[Prop. 3.3]{arnoldFEEC}.
\end{remark}
We distinguish the adjoint corresponding to the unweighted inner product and the adjoint corresponding to the weighted inner product, and we write $D^{*, 1}_{k}$ to emphasize the unweighted adjoint when necessary.

Similarly to how we define the adjoint of the exterior derivative, the adjoint of the difference operator is defined through the equation $\langle \delta \alpha, \beta \rangle_{\Ac^{p,q}} = \langle \alpha, \delta^* \beta \rangle_{\Ac^{p-1,q}}$. We can describe the adjoint of the difference operator $\delta$ with the characteristic function. 

\begin{proposition}\label{prop: total diff adj}
The adjoint of the total differential takes the following form:
\begin{equation}
D^{*,w}_k = d^{*,w} - (-1)^k \delta^{*,w}.  
\end{equation}
\end{proposition}
\begin{proof}
The result follows immediately from the linearity of the inner product and the definition of the adjoint
\begin{align*}
\langle \alpha,  D^{*,w}_{k} \beta \rangle_{\Ac^{k-1}_w}  &= \langle D^{k-1} \alpha, \beta \rangle_{\Ac^k_w} \\
 &= \langle d\alpha, \beta \rangle_{\Ac^k_w} + \langle  (-1)^{k-1} \delta \alpha, \beta \rangle_{\Ac^k_w} \\
 &= \langle \alpha, d^{*,w} \beta \rangle_{\Ac^{k-1}_w} + \langle \alpha, (-1)^{k-1} \delta^{*,w} \beta \rangle_{\Ac^{k-1}_w} \\
 &= \langle \alpha, d^{*,w} - (-1)^{k} \delta^{*,w} \beta \rangle_{\Ac^{k-1}_w}.
\end{align*}    
\end{proof}

Following \Cref{sub: exterior derivative HL}, the unweighted and weighted Hodge-Laplace operators for the \v{C}ech-de Rham complex are defined as
\begin{align}
    \Delta_{D,1}^{k} &:= D^{k - 1}D_k^{*,1} + D_{k + 1}^{*,1} D^k,\\
    \Delta_{D,w}^k &:= D^{k - 1}D_{k}^{*,w} + D_{k + 1}^{*,w} D^k.
\end{align}
While the weights are important in applications, from the perspective of analysis the distinction is typically not critical, and when no confusion arises, we will in the continuation omit both the weight and the degree $k$ in the notation of $D$, $D^*$, and $\Delta_D$. 
The unweighted and weighted Hodge-Laplace problem are then both defined as follows: Given $\varphi \perp \ker \Delta_{D}$, find $\alpha \in \dom \Delta_{D}$ such that:
\begin{align} \label{eq: total HL}
    \Delta_{D} \alpha &= \varphi, &
    \alpha &\perp \ker \Delta_{D}.
\end{align}

The operators $\Delta_{D, w}$ and $\Delta_{D,1}$ can be explicitly related using the knowledge of the exterior derivative $d$ and the difference operator $\delta$, according to the following two propositions.
\begin{proposition}
The weighted Hodge-Laplacian on the \v{C}ech-de Rham complex decomposes to the weighted Hodge-Laplacians of $d$ and $\delta$ plus four coupling terms:
\begin{equation}\label{eq: weighted hodge-laplacian}
    \Delta_{D,w} = \Delta_{d,w} + \Delta_{\delta,w}+ (-1)^k (d^{*,w} \delta 
    - \delta d^{*,w}
    + \delta^{*,w} d
    - d \delta^{*,w}).
\end{equation}
\end{proposition}
\begin{proof}
We substitute the definition of $D^k$ and the expression for the weighted adjoint $D^{*,w}_k$ given in \cref{prop: total diff adj} into the definition of the Hodge-Laplace operator to get
\begin{align*}    
\Delta_{D,w}^k &= D^{*,w}_{k+1} D^k + D^{k-1} D^{*,w}_k \\
&= \left( d^{*,w} + (-1)^k \delta^{*,w} \right) (d + (-1)^k \delta) 
+ (d + (-1)^{k-1} \delta) (d^{*,w} + (-1)^{k-1}\delta^{*,w} ) \\
&= d^{*,w}d + d d^{*,w} + \delta^{*,w} \delta + \delta \delta^{*,w} 
+ (-1)^k (d^{*,w} \delta + \delta^{*,w} d - d \delta^{*,w} - \delta d^{*,w}).
\end{align*}  
Reordering terms and identifying $d^{*,w}d + d d^{*,w}$ with $\Delta_{d,w}$ and $\delta^{*,w} \delta + \delta \delta^{*,w}$ with $\Delta_{\delta, w}$, we arrive at
\begin{equation*}
    \Delta_{D,w} = \Delta_{d,w} + \Delta_{\delta,w}+ (-1)^k (d^{*,w} \delta 
    - \delta d^{*,w}
    + \delta^{*,w} d
    - d \delta^{*,w}).
\end{equation*}
\end{proof}

The unweighted Hodge-Laplacian takes a simpler form, as pointed out in the next corollary.
\begin{corollary} \label{prop: decomposition HL}
The unweighted Hodge-Laplacian on the \v{C}ech-de Rham complex decomposes to the sum of the unweighted Hodge-Laplacians of $d$ and $\delta$:
\begin{equation}
    \Delta_{D,1} = \Delta_{d,1} + \Delta_{\delta,1}.
\end{equation}
\end{corollary}
\begin{proof}
For $\alpha \in \dom(\Delta_{D,1}^k)$, we have that $d\delta^*\alpha  =\delta^* d \alpha$ as well as $d^*\delta \alpha = \delta d^* \alpha$, and hence the coupling terms in \cref{eq: weighted hodge-laplacian} cancel out, and we are left with the stated corollary.
\end{proof}

The diagram below illustrates how the different components of the weighted Hodge-Laplacian acts on differential forms of degree $(p,q)$. The vertical and horizontal maps are the Hodge-Laplacians $\Delta_{d,w}$ and $\Delta_{\delta,w}$, respectively, which increases or decreases the degree of $q$ or $p$ by one, then maps back to $\Ac^{p,q}$. The coupling terms preserves the total degree $k$, but  with codomain $\Ac^{p-1, q+1}$ and $\Ac^{p+1, q-1}$.

\begin{equation}
\begin{tikzcd}[column sep=scriptsize]
	{\mathcal{A}^{p-1,q+1}} && \bullet \\
	\\
	\bullet && {\mathcal{A}^{p,q}} && \bullet \\
	\\
	&& \bullet && {\mathcal{A}^{p+1,q-1}}
	\arrow["{(-1)^k(\delta^{*,w}d - d \delta^{*,w})}"{pos=0.8}, from=3-3, to=1-1]
	\arrow["{(-1)^k(d^{*,w}\delta - \delta d^{*,w})}"{pos=0.8}, from=3-3, to=5-5]
	\arrow["{dd^{*,w}}"', leftrightarrow, from=3-3, to=5-3]
	\arrow["{d^{*,w}d}"', leftrightarrow, from=3-3, to=1-3]
	\arrow["{\delta^{*,w}\delta}"', leftrightarrow, from=3-3, to=3-5]
	\arrow["{\delta\delta^{*,w}}"', leftrightarrow, from=3-3, to=3-1]
\end{tikzcd}   
\end{equation}

\subsection{Results from Hilbert complex theory}
In the following section we show that the $L^2$ \v{C}ech-de Rham complex is not just a Hilbert complex, but it also satisfies the compactness property, which in turn implies that the range of $D^k$ is closed in $L^2 \Ac^{k+1}$. As a consequence, general results for Hilbert complexes with the compactness property can now be directly applied to all coupled problems that can be identified as a Hodge-Laplace problem on a \v{C}ech-de Rham complex.

\begin{theorem} \label{thm: CdR has the compactness property}
The $L^2$ \v{C}ech-de Rham complex has the \emph{compactness property}, i.e. the inclusion $\dom(D^k) \cap \dom(D^*_k) \subset L^2 \Ac^k$ is compact for all $k$.
\end{theorem}
\begin{proof}
As noted in \Cref{prop: CDUL}, $\dom D^{p,q} = \dom d^q$ and by analogous arguments, we have $\dom D^*_{p,q} = \dom d^*_q$ and thus, $\dom D^{p,q} \cap \dom D^*_{p,q} = \dom d^q \cap \dom d^*_q$, which is compactly embedded in each $L^2 \Ac^{p, q}$, see \cite{pauly2022hilbert1, picard1984}. The compact embedding in $L^2 \Ac^k$ then follows by direct summation.
\end{proof}

We now summarize some general results for Hilbert complexes with the compactness property in the following corollaries, with the understanding that for any of the concrete examples provided in \Cref{sec: examples}, sharper results may be obtained by exploiting problem-specific aspects of the examples.

\begin{corollary}[{\cite[Thm. 4.4]{arnoldFEEC}}]
The $L^2$ \v{C}ech-de Rham complex is a \emph{closed Hilbert complex}, meaning that the range of $D^k$ is closed in $L^2 \Ac^{k+1}$. Moreover, the $L^2$ \v{C}ech-de Rham complex is a \emph{Fredholm complex}, i.e. it has finite-dimensional cohomology. 
\end{corollary}
\begin{corollary}[{\cite[Thm. 4.5]{arnoldFEEC}}] \label{cor: Hodge decomposition}
    For each $k$, we have the orthogonal \emph{Hodge decomposition}:
    \begin{equation}
    L^2 \Ac^k = \im D^{k-1}  \oplus \ker \Delta_D^k \oplus \im D^*_{k+1}.
    \end{equation}
\end{corollary}

\begin{corollary}[{\cite[Thm. 4.6]{arnoldFEEC}}] \label{cor: Poincaré inequality}
    For each $k$, there exists a constant $C_k$ such that the \emph{Poincaré} inequality holds:
    \begin{align}
        \|\alpha\|_{\Ac^k} &\leq C_k \|D^k \alpha \|_{\Ac^{k+1}}, &
        \forall \alpha &\in \dom D^k \cap (\ker D^k)^\perp.
    \end{align}
\end{corollary}

A major consequence of \Cref{thm: CdR has the compactness property} is that all systems of equations corresponding to a Hodge-Laplace problem are well-posed, as shown in the following corollary. 
\begin{corollary}[{\cite[Thm. 4.8]{arnoldFEEC}}] \label{cor: Well-posedness}
    The Hodge-Laplace problem on the \v{C}ech-de Rham complex \eqref{eq: total HL} is well-posed.
\end{corollary}

\subsection{Solution theory of evolutionary equations}
In this section, we generalize the Hodge-Laplace problem \eqref{eq: total HL} by introducing a time-dependent term. In particular, let us consider the following problem: find $\alpha$ such that
\begin{align} \label{eq: timedependent}
    \partial_t^\ell \alpha + \Delta_D \alpha &= \phi, &
    \ell &\in \{1, 2\}.
\end{align}

For $\ell = 1$, we refer to this problem as the parabolic Hodge-heat equation. On the other hand, for $\ell = 2$, we obtain the hyperbolic Hodge-wave equation. Since we aim to analyze both problems in a single framework, we turn to the solution theory for evolutionary equations \cite{picard2015well}. 

We start by introducing several preliminary definitions. First, let $V$ be the Hilbert space given by
\begin{align} \label{eq: def V}
    V &:= L^2 \Ac^k \times L^2\Ac^{k + 1} \times L^2\Ac^{k - 1}.
\end{align}
Let the inner product and norm on $V$ be defined as follows for $v = (\alpha, \beta, \gamma) \in V$ and $\tilde{v} = (\tilde{\alpha}, \tilde{\beta}, \tilde{\gamma}) \in V$:
\begin{align}
    \langle v, \tilde{v} \rangle_V &:= 
    \langle \alpha, \tilde{\alpha} \rangle_{\Ac^k_w}
    + \langle \beta, \tilde{\beta} \rangle_{\Ac^{k + 1}_w}
    + \langle \gamma, \tilde{\gamma} \rangle_{\Ac^{k - 1}_w}, &
    \| v \|_V^2 &:= \langle v, v \rangle_V.
\end{align}

Next, for given parameter $\mu$, we define the weighted Bochner space and its associated norm as follows:
\begin{align*}
    L_\mu^2(\mathbb{R}; V) &:= \left\{ g: \mathbb{R} \to V \mid \| g \|_{L_\mu^2(\mathbb{R}; V)} < \infty \right\}
    , &
    \| g \|_{L_\mu^2(\mathbb{R}; V)}
    &:= \int_{\mathbb{R}} e^{- \mu t} \| g(t) \|_V^2 dt.
\end{align*}
On this space, the temporal derivative is defined as
\begin{align}
    \partial_{0, \mu} &: \dom(\partial_{0, \mu}) \subseteq L_\mu^2(\mathbb{R}; V) \to L_\mu^2(\mathbb{R}; V), &
    \partial_{0, \mu} &:= e^{\mu t} (\partial_t + \mu) e^{- \mu t}.
\end{align}
We emphasize that $\partial_{0, \mu} g = \partial_t g$ for differentiable $g: \mathbb{R} \to \mathbb{R}$, which can be confirmed by direct calculation.

Finally, for an operator $A: \dom(A) \subseteq V \to V$, let $\Gc(A) := \{ (v, Av), v \in \dom(A) \}$ be its graph. $\overline{A}$ then denotes the closure of $A$ which, if it exists, satisfies $\Gc(\overline{A}) = \overline{\Gc(A)}$.

With these prerequisites in place, we state the following key result for linear evolutionary equations.

\begin{theorem}[{\cite[Thm. 2.5]{picard2015well}}] \label{thm: Picard}
    Let $\mu > 0$ and $r > \frac1{2\mu}$. Let $A: \dom(A) \subseteq V \to V$ and let $M(z): V \to V$ be a continuous linear operator for all $z$ in the open disk $B_{\mathbb{C}}(r, r) \subset \mathbb{C}$ with center $r$ and radius $r$. Assume that:
\begin{enumerate}[label = {A\arabic*.}, ref = {A\arabic*}]
        \item The operator $A$ is skew-adjoint. \label{ass: A1}
        \item There exists $c>0$ such that $z^{-1} M(z) - cI$ is monotone for all $z \in B_{\mathbb{C}}(r, r)$. \label{ass: A2}
\end{enumerate}

    Then, for a given $f \in L_\mu^2(\mathbb{R}; V)$, a unique $v \in L_\mu^2(\mathbb{R}; V)$ exists that satisfies the evolutionary equation
    \begin{align} \label{eq: Picard format}
        \left(\overline{\partial_{0, \mu}M(\partial_{0, \mu}^{-1}) + A}\right) v = f.
    \end{align}
    Moreover, $v$ satisfies the bound
    $c \| v \|_{L_{0, \mu}^2(\mathbb{R}; V)} \le \| f \|_{L_{0, \mu}^2(\mathbb{R}; V)}$.
\end{theorem}

We are now ready to fit the Hodge-heat and Hodge-wave equations \eqref{eq: timedependent} in the framework of evolutionary equations and show that these problems are well-posed.

\begin{corollary} \label{cor: Hodge-heat}
    For $\mu > 0$ and $\phi \in L_{0, \mu}^2(\mathbb{R}; L^2 \Ac^k)$, the Hodge-heat problem \eqref{eq: timedependent} with $\ell = 1$ admits a unique solution $\alpha \in L_{0, \mu}^2(\mathbb{R}; L^2 \Ac^k)$ that satisfies
    \begin{align} \label{eq: result heat}
        c
        \| (\alpha, D^k \alpha, D_k^* \alpha) \|_{L_{0, \mu}^2(\mathbb{R}; V)}
        \le \| \phi \|_{L_{0, \mu}^2(\mathbb{R}; L^2 \Ac^k)},
    \end{align}
    for some $c > 0$.
\end{corollary}
\begin{proof}
    We cast the problem in the format of evolutionary equations \eqref{eq: Picard format} and then invoke the well-posedness result from \Cref{thm: Picard}.     
    Let $\beta = D^k \alpha$ and $\gamma = D_k^* \alpha$, so that the problem is rewritten as: find $(\alpha, \beta, \gamma) \in V$ such that
    \begin{align} \label{eq: 3field heat}
        \left(
        \partial_{0, \mu}
        \begin{bmatrix}
            1 \\ & 0 \\ && 0
        \end{bmatrix}
        +
        \begin{bmatrix}
            0 \\ & 1 \\ && 1
        \end{bmatrix}
        +
        \begin{bmatrix}
            & D_{k + 1}^* & D^{k - 1}\\
            - D^k \\
            - D_k^*
        \end{bmatrix}
        \right)
        \begin{bmatrix}
            \alpha \\ \beta \\ \gamma
        \end{bmatrix}
        &= 
        \begin{bmatrix}
            \phi \\ 0 \\ 0
        \end{bmatrix}.
    \end{align}

    Problem \eqref{eq: 3field heat} has the form \eqref{eq: Picard format} with $v := (\alpha, \beta, \gamma)$, $f := (\phi, 0, 0)$, and
    \begin{align} \label{eq: def M A}
        M(z) &:= \begin{bmatrix}
            1 \\ & 0 \\ && 0
        \end{bmatrix}
        +
        z
        \begin{bmatrix}
            0 \\ & 1 \\ && 1
        \end{bmatrix},
        &
        A &:= 
        \begin{bmatrix}
            & D_{k + 1}^* & D^{k - 1}\\
            - D^k \\
            - D_k^*
        \end{bmatrix}.
    \end{align}
    
    We now confirm the two assumptions of \Cref{thm: Picard}. 
    First, \ref{ass: A1} is valid by the definition of the total codifferential \eqref{eq: total codiff}. For the second, we note that $z \in B_{\mathbb{C}}(r, r)$ implies that $\Re(z^{-1}) > \frac1{2 r}$. In turn, we derive
    \begin{align*}
        \Re\left(\langle v, z^{-1} M(z) v \rangle_V\right)
        = \Re\left(\left\langle v,
        \begin{bmatrix}
            z^{-1} \\ & 1 \\ && 1
        \end{bmatrix} v \right\rangle_V \right)
        \ge \min\left\{ 1, \frac1{2 r}\right\} 
        \langle v, v \rangle_V.
    \end{align*}
    By linearity of $M(z)$, assumption \ref{ass: A2} now follows with $c = \min\left\{ 1, \frac1{2 r}\right\} > 0$. \Cref{thm: Picard} then implies that a unique $v = (\alpha, \beta, \gamma)$ exists, bounded in the $L^2_{0, \mu}(\mathbb{R}; V)$-norm. In turn, the result \eqref{eq: result heat} follows by equivalence of the problems \eqref{eq: timedependent} and \eqref{eq: 3field heat}.
\end{proof}

\begin{corollary} \label{cor: Hodge-wave}
    For $\mu > 0$ and $\phi \in L_{0, \mu}^2(\mathbb{R}; L^2 \Ac^k)$ with $\partial_{0, \mu}^{-1} \phi \in L_{0, \mu}^2(\mathbb{R}; L^2 \Ac^k)$, the Hodge-wave problem \eqref{eq: timedependent} with $\ell = 2$ admits a unique solution $\alpha \in L_{0, \mu}^2(\mathbb{R}; L^2 \Ac^k)$ that satisfies
    \begin{align}
        c \| (\alpha, \partial_{0, \mu}^{-1} D^k \alpha, \partial_{0, \mu}^{-1} D_k^* \alpha) \|_{L_{0, \mu}^2(\mathbb{R}; V)}
        \le \| \partial_{0, \mu}^{-1} \phi \|_{L_{0, \mu}^2(\mathbb{R}; L^2 \Ac^k)},
    \end{align}
    for some $c > 0$.
\end{corollary}
\begin{proof}
    The proof proceeds similar to \Cref{cor: Hodge-heat}. Let $V$ be defined as in \eqref{eq: def V} and let now $\beta = \partial_{0, \mu}^{-1} D^k \alpha$ and $\gamma = \partial_{0, \mu}^{-1} D_k^* \alpha$. The problem then becomes: find $(\alpha, \beta, \gamma) \in V$ such that
    \begin{align} \label{eq: 3field wave}
        \left(
        \partial_{0, \mu} I
        +
        A
        \right)
        \begin{bmatrix}
            \alpha \\ \beta \\ \gamma
        \end{bmatrix}
        &= 
        \partial_{0, \mu}^{-1}
        \begin{bmatrix}
            \phi \\ 0 \\ 0
        \end{bmatrix},
    \end{align}
    with $A$ as in \eqref{eq: def M A}. Again, the first assumption \ref{ass: A1} follows directly from \eqref{eq: total codiff}. Here, \ref{ass: A2} is also immediate with $c = \frac1{2r}$ since $M(z) := I$. To obtain the result on $\alpha$, we use the equivalence between the problems \eqref{eq: timedependent} and \eqref{eq: 3field wave}.
\end{proof}

A more general class of nonlinear problems involving maximal monotone operators can be analyzed by following \cite[Thm. 3.2]{picard2015well}. We have restricted our exposition to the linear case for conciseness.

\subsection{Further implications for approximation methods methods}
Finally, we mention that the framework of finite element exterior calculus \cite{afwFEEC3} directly provides stable and convergent mixed finite element discretizations for these problems, given a conforming, simplicial grid. Additionally, functional guaranteed a posteriori bounds can be obtained using the techniques summarized in \cite{pauly2020solution}.

\section{Examples in mathematical modeling} 
\label{sec: examples}\noindent 
We consider three elementary examples in which $\Uc = \{ U_0, U_1 \}$ forms an open cover of a given domain $\Omega$, as illustrated in \Cref{fig: figure1}. In each case, we show that the Hodge-Laplace problem corresponds to the coupled equations that arise in physical applications.
\begin{figure}[htb]
    \centering
    \includegraphics[width=0.9\linewidth]{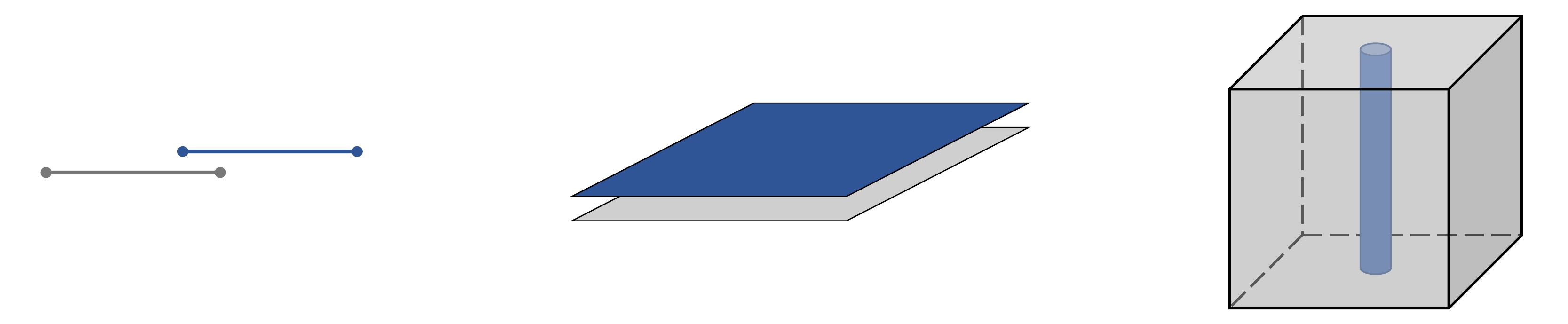}
    \caption{The open covers considered in the three examples. From left to right, the total domain $\Omega$ consists of a line segment, a square, and a cube, respectively. In each case, $U_0$ is illustrated in gray and $U_1$ in blue.}
    \label{fig: figure1}
\end{figure}

We focus on $\Delta_{D,w}^0$ i.e. the weighted Hodge-Laplace operator defined on $L^2 \Ac^0 = L^2 \Ac^{0,0}$ in the bottom-left corner of the double complex \eqref{eq: double complex}. To make the exposition accessible, we will use the canonical proxies of differential forms, through the identification of ``scalar-valued'' spaces $L^2\Lambda^0(U_i) 
\cong L^2(U_i)$, as well as ``vector-valued'' spaces $L^2\Lambda^1(U_i) \cong 
(L^2(U_i))^n$.

In particular, we will denote proxies by Latin letters: each 0-form $\alpha_i \in L^2 \Lambda^0(U_i)$ is identified by a scalar proxy $a_i \in L^2(U_i)$ and each 1-form $\beta_i \in L^2\Lambda^1(U_i)$ can be identified by a vector proxy $b_i \in (L^2(U_i))^n$. 

On $L^2 \Ac^0$, the differential and difference operators then correspond to $d\alpha_i = \nabla a_i$ and $\delta \alpha = (a_1 - a_0)|_{U_{0, 1}}$, respectively. These definitions lead to the following diagram:
\begin{equation} \label{eq: diagram example}
    \begin{tikzcd}
        (b_0, b_1) \arrow[r, mapsto, "- \delta"] & g_{0, 1} \\
        (a_0, a_1) \arrow[r, mapsto, "\delta"]\arrow[u, mapsto, "\nabla"] & b_{0, 1} \arrow[u, mapsto, "\nabla"].
    \end{tikzcd}
\end{equation}

When considering weighted inner products as defined in Section 2.1, we will consistently use the convection that we scale the problem to accept unit weights for the inner product an $\Ac^0$, while we apply weighted inner products on $\Ac^1$, denoted $w_0$, $w_1$ and $w_{0,1}$. The adjoint operator for $D^0$ is then calculated as 
\begin{align*} \label{eq: simple codiff}
    \langle D_{1}^{*}\beta, \alpha \rangle_{\Ac^0} &= 
    \langle \beta, D^0 \alpha \rangle_{\Ac^1_w}\\
    &= \langle w_0 b_0, \nabla a_0 \rangle_{U_0} + 
    \langle w_1 b_1, \nabla a_1 \rangle_{U_1}
    + \langle w_{0,1} b_{0,1}, (a_1-a_0) \rangle_{U_{0,1}} \\
    &= \langle -\nabla_0 \cdot (w_0 b_0) - \mathbbm{1}_{0, 1}^0 (w_{0,1} b_{0,1}), a_0 \rangle_{U_0} \\ 
    &+ \langle -\nabla_0 \cdot (w_1 b_1) + \mathbbm{1}_{0, 1}^1 (w_{0,1} b_{0,1}),  a_1 \rangle_{U_1}\\
    &+ \sum_i \int_{\partial U_i} a_0 w_{0,1} b_{0,1}\cdot dA
\end{align*}
We limiting the domain of the codifferential so that the surface integrals vanish.  
Thus the codifferential $D_{1}^{*}$, expressed on the proxies, is composed of the divergence operator restricted to functions with homogeneous boundary conditions $\nabla_0\cdot$, and the indicator function of the overlap $U_{0, 1} \subset U_i$ denoted $\mathbbm{1}_{0, 1}^i$:
\begin{align}
    D^*_{1}\beta &= - \nabla_0 \cdot (w_i b_i) + (-1)^{i+1} \mathbbm{1}_{0, 1}^i (w_{0,1} b_{0, 1}), &
    \text{on }&U_i.
\end{align}

Composing the differential and codifferential operators, he Hodge-Laplace problem becomes as follows for given $\phi = \{ f_0, f_1\} \in L^2 \Ac^0$: Find $\alpha \in \dom \Delta_{D}^0 \subset L^2 \Ac^0$ such that $\Delta_{D}^0 \alpha = \phi$. Written using proxies, the Hodge-Laplace problem becomes
\begin{subequations} \label{eq: primal HL}
\begin{align}
-\nabla \cdot (w_0\nabla a_0) - \mathbbm{1}_{0, 1}^0 w_{0,1}(a_1 - a_0) &= f_0, &
    \text{on } &U_0, \\
-\nabla \cdot (w_1\nabla a_1) + \mathbbm{1}_{0, 1}^1 w_{0,1}(a_1 - a_0) &= f_1, &
    \text{on } &U_1. 
\end{align}
The domain of the adjoint differential imposes the following boundary conditions:
\begin{align} \label{eq: bc primal}
    \nu \cdot (w_i \nabla a_i) &= 0, & \text{on } &\partial U_i, 
\end{align}
in which $\nu$ is the outward unit normal of $\partial U_i$. Finally, the kernel of the Hodge-Laplacian, $\ker \Delta_{D}$ is given by the constants, thus orthogonality to the kernel imposes the constraint
\begin{align} \label{eq: constraint primal}
\langle a_0,1 \rangle_{U_0} + 
    \langle a_1,1 \rangle_{U_1}=0.
\end{align}
\end{subequations}

We can similarly state the mixed formulation of the Hodge-Laplace problem \eqref{eq: primal HL} in terms of the canonical proxies: Find $(\alpha, \beta) \in \dom D^0 \oplus \dom D_{1}^* \subset L^2 \Ac^0 \oplus L^2 \Ac^1$ such that
\begin{subequations} \label{eqs: example mixed}
\begin{align}
    b_0 = \nabla a_0,& \qquad
    b_1 = \nabla a_1, \qquad
    b_{0, 1}= a_1 - a_0, \label{eq: constitutive law} \\
    - \nabla \cdot (w_0 b_0) - \mathbbm{1}_{0, 1}^0 (w_{0,1} b_{0, 1}) &= f_0,\qquad  
    - \nabla \cdot (w_1 b_1) + \mathbbm{1}_{0, 1}^1 (w_{0,1} b_{0, 1}) = f_1, \label{eq: conservation law}
\end{align}    
subject to the boundary conditions   
\begin{align}
    \nu \cdot (w_i b_i) &= 0, & \text{on } & \partial U_i 
    . \label{eq: bc mixed}
\end{align}
\end{subequations}
The same constraint, equation \eqref{eq: constraint primal}, as in the primal formulation still holds. 

\begin{remark}\label{rem: non0}
Problem \eqref{eqs: example mixed} is one of $(n + 2)$ Hodge-Laplace equations arising from the cover $\Uc = \{U_0, U_1\}$, since the order $k$ ranges from zero to $(n + 1)$.
Using \Cref{prop: decomposition HL}, the (unweighted) Hodge-Laplace operator for $\gamma \in L^2 \Ac^k$ with $1 \le k \le n$ can be written in terms of proxies as
\begin{align*}
    \Delta_D^k \gamma &= - \Delta g_i + \mathbbm{1}_{0, 1}^i (g_i - g_j) & \text{on } &U_i, \ i + j = 1, \\
    \Delta_D^k \gamma &= - \Delta g_{0, 1} + 2 g_{0, 1} & \text{on } &U_{0, 1}.
\end{align*}
Here, $-\Delta g_i$ denotes the vector-Laplacian $\nabla \times \nabla \times g_i -\nabla (\nabla \cdot g_i)$ if the proxy $g_i$ is vector-valued ($n=3$ and $k\in (1,2)$). For $\gamma \in L^2 \Ac^{n + 1}$ at the end of the complex, we derive:
\begin{align*}
    \Delta_D^{n + 1} \gamma &= - \Delta g_{0, 1} + 2 g_{0, 1} & \text{ on } &U_{0, 1}.
\end{align*}
All $(n + 2)$ distinct Hodge-Laplace problems are well-posed (subject to appropriate boundary conditions and orthogonality constraints as given above) due to \Cref{cor: Well-posedness}.
\end{remark}

\begin{remark}\label{rem: multiple}
If the open cover contains more sets, i.e. $\Uc = \{U_0, ..., U_N\}$, then the weighted Hodge-Laplace operator on $L^2 \Ac^0$ is given by:
\begin{align*}
    \Delta_{D}^0 \alpha &
    = - \nabla\cdot(w_i \nabla a_i) + \sum_{j \neq i} \mathbbm{1}_{i, j}^i w_{j,i}(a_i - a_j) 
    & \text{on } &U_i, \ i \in \Ic,
\end{align*}
using the convention $\mathbbm{1}_{i, j}^i = - \mathbbm{1}_{j, i}^i$.
\end{remark}

\subsection{Two joined, elastic rods in 1D}
\label{sub: 1D example}\noindent
We start with the geometry illustrated in the left of \Cref{fig: figure1}, in which the dimension $n=1$, and the sets $U_0 := (-1, \epsilon)$ and $U_1 := (-\epsilon, 1)$ for some $0 < \epsilon < 1$ form an open cover of $\Omega = (-1, 1)$. This case can be physically realized as illustrated in Figure \ref{fig: rods}, where the model for displacement and strain correspond to the $k=0$ Hodge-Laplace problem. 

\begin{figure}[htb]
    \centering
    \includegraphics[width=0.4\linewidth]{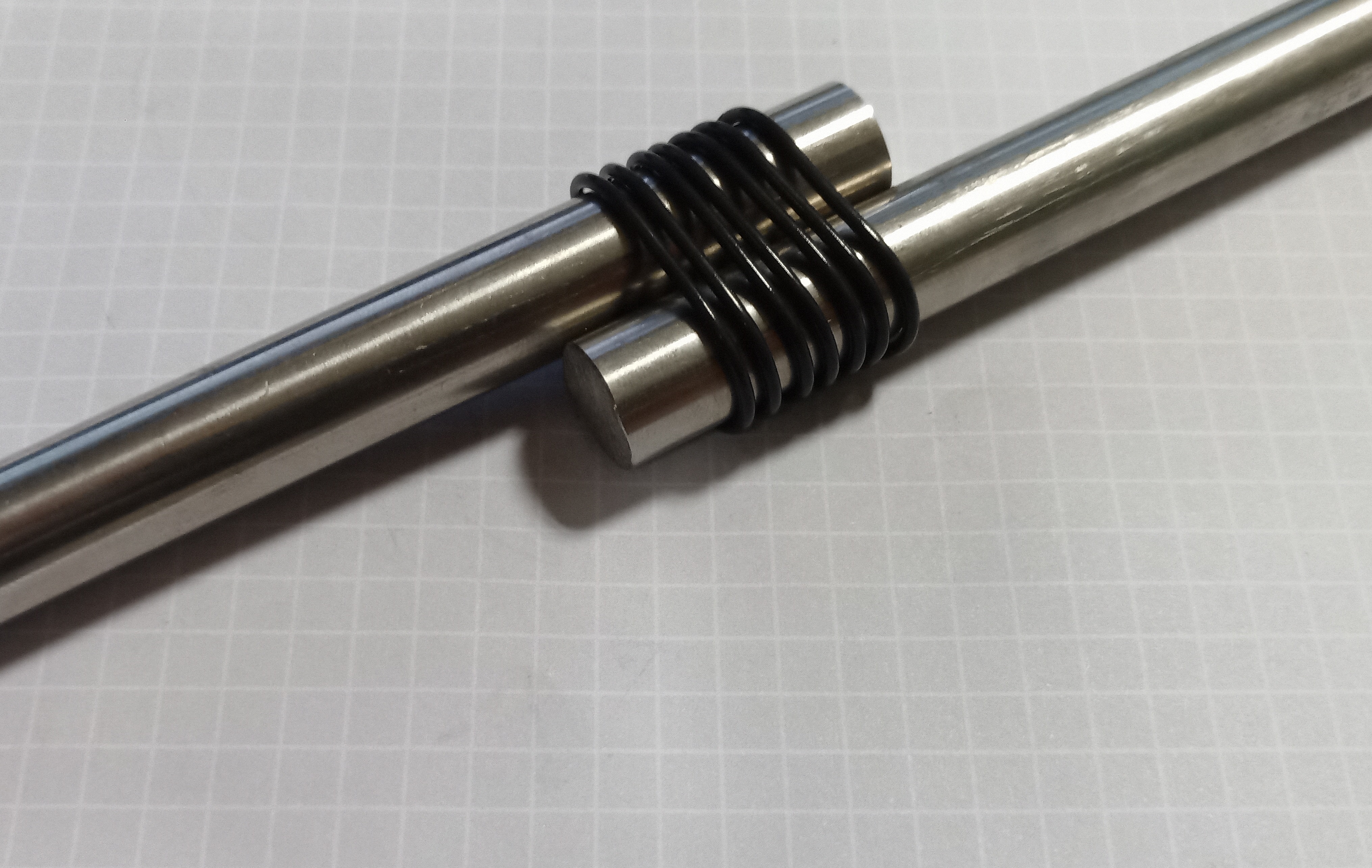}
    \caption{Stainless steel rods joined using a small overlap. The joining force is dependent the relative displacement of the two rods (due to tensioning of the lashing)}
    \label{fig: rods}
\end{figure}

For this case, the system \eqref{eqs: example mixed} describes two linearly elastic rods that are elastically connected. As is clear from equation \eqref{eq: constitutive law}, the variables $a_i$ and $b_i$ model the displacement and strain, respectively, and $U_i$ is the initial domain of rod $i$, with $i \in \{ 0, 1 \}$. A third strain variable $b_{0, 1}$ on the overlap $U_{0, 1}$ captures the elongation of the connecting welding (for illustrative purposes represented by a lashing using elastic strings in tension).

Equation \eqref{eq: conservation law} includes Hooke's law in the sense that the linearly weighted strain variables correspond to elastic stress. Similarly, the strain $b_{0, 1}$ may in general be considered monotonically related to elastic stress associated with the stretching the connecting strings. In a linearized regime, we also here obtain proportionality between displacement difference (discrete strain) and stress. Finally, \eqref{eq: conservation law} also includes the momentum balance due to the presence of the divergence operator. 
The boundary conditions \eqref{eq: bc mixed} imply tension-free conditions on both ends of both rods, while the constraint \eqref{eq: constraint primal} defines the mean position of two rods to be at the origin on the real line. 
Forces acting on the two rods are incorporated in the right-hand side terms $f_i$.


This model can be extended by including the term $\partial_t^2 \alpha$ on the left-hand side to model longitudinal acceleration of the rods in time. In particular, the equation $\partial_t^2 \alpha + \Delta_D \alpha = \phi$ then allows for wave propagation along the joined rods.

\subsection{Multiple continuum models of porous materials}
\label{sub: multi-porosity}\noindent
Let us consider the example illustrated in the middle of \Cref{fig: figure1}, in which $\Uc$ is an open cover of $\Omega \subset \mathbb{R}^2$ with $U_0 = U_1 = \Omega$. 

\begin{figure}[htb]
    \centering
    \includegraphics[width=1.0\linewidth]{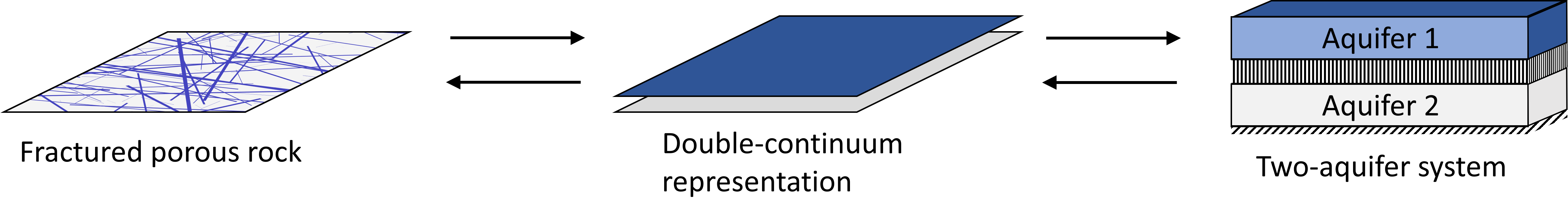}
    \caption{Left and middle: Example of a fractured porous medium represented as two fully overlapping domains. Right and middle: Example of a system with two aquifers, each modeled as two-dimensional, connected by a low-permeable aquitard, again represented as two fully overlapping domains.}
    \label{fig: fracture}
\end{figure}

We provide two physical interpretation of this model. In the left part of Figure \ref{fig: fracture}, the model corresponds to one where two coupled elliptic variables are used to model a process on the whole domain, such as is common for so-called "double-continuum" models of fractured porous media \cite{barenblatt1960basic,chen1989transient,straughan2017mathematical}. A characteristic illustration of such a material is given in Figure \ref{fig: fracture}.  The physical interpretation of this modeling concept is that for percolating fracture networks, the timescales of flow in the fracture network and porous rock separate, such that it is appropriate to consider two fluid pressures: One representing the pressure in the fractures and one representing the pressure in the rock. In a homogenized model, each of these pressures are defined on the full domain, leading to the double-continuum concept \cite{arbogast1990derivation}. 

Concretely, the Hodge-Laplace problem \eqref{eqs: example mixed} describes exactly the double porosity model for this problem as formulated in the above references. We identify that as the domains fully overlap, then $U_0 = U_1 = U_{0,1}$ and thus $\mathbbm{1}_{0,1} \equiv 1$. Furthermore, $a_i$ describes the fluid pressure inside porous medium compartment $i$, the variable $b_i$ corresponds to the driving force for fluid motion (the gradient of pressure) within compartment $i$ and $b_{0, 1}$ driving force for exchange between $U_0$ and $U_1$ (the pressure difference). As is clear from the material properties included in equation \eqref{eq: conservation law}, the governing equations are given proportionality between driving force and fluid flow (known as Darcy's law) and the mass (or volume) balance equations. The boundary conditions \eqref{eq: bc mixed} describe zero fluid flux across the boundary, while the constraint \eqref{eq: constraint primal} fixes the mean value of the pressure across the system. Fluid compressibility can be incorporated in the system by including the time-dependent term $\partial_t \alpha$. 

Returning to Figure \ref{fig: fracture}, the double-porosity model is mathematically identical to the models of two interconnected aquifers separated by an aquitard, as studied in the hydrology literature \cite{motz1978steady, hunt1985flow, bear2013dynamics} 
For the case of two aquifers, this system is directly included in the framework as described herein. For the case of more than two aquifers, where the aquifers and aquitards form an alternating stack, this system is a degenerate limit of our exposition, which is obtained by letting the weights $w_{i,j}$ in Remark \ref{rem: multiple} equal zero whenever $i-j>1$.  

Finally, we emphasize that the so-called multiple-network models discussed in e.g. \cite{lee2019mixed}, used to model flow of extravascular fluids in the brain, are also structurally identical to the equations given in Remark \ref{rem: multiple}.

\subsection{Mixed-dimensional coupling with high dimensionality gap}\noindent
Our third example is illustrated by the sketch to the right in Figure \ref{fig: figure1}. Here  the cover $\Uc$ corresponds to a three-dimensional domain $\Omega$ such that $U_0 := \Omega$, while $U_1 \subset U_0$ is an embedded, vertical cylinder with small radius $\epsilon > 0$.  Such models commonly arise for thin inclusions (such as fiber reinforced materials), and a particular topical example is that of blood vessels within a tissue, as illustrated in Figure \ref{fig: brain}. 

\begin{figure}[htb]
    \centering
\includegraphics[width=0.9\linewidth]{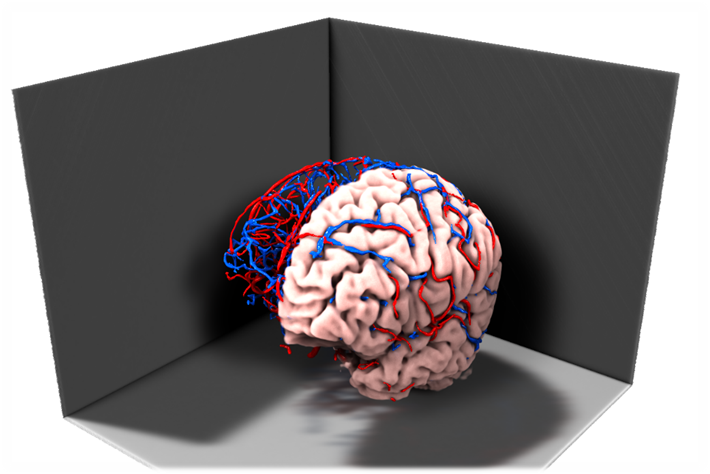}
    \caption{Blood flow in living tissue, such as this figure of the human brain, is frequently modeled as mixed-dimensional. In such a setting, the veins and arteries that are big enough to be resolved (visualized as blue and red, respectively) are modeled as graphs of one-dimensional segments, while the remaining tissue (is modeled as a three-dimensional domain). Figure from \cite{doi:10.1137/20M1362541}.}
    \label{fig: brain}
\end{figure}

As in the previous example, we interpret $a_0$ and $a_1$ as fluid pressures in the surroundings and the cylinder, respectively, $b_0, b_1$ are the corresponding fluxes, and $b_{0, 1}$ describes the mass exchange between the cylinder and the bulk. We recognize the material parameters as the permeability in the bulk ($w_0$), the resistance appearing in the Hagen-Poiseuille law for flow in the cylinder ($w_1\sim \epsilon^4$), and an exchange coefficient, associated with filtration between the cylinder and the bulk proportional to a pressure difference.  
The system \eqref{eqs: example mixed} therefore governs porous medium flow in subsurface systems with an injection or production well. The same equations are encountered in blood perfusion models of vascularized, biological tissue, see e.g. \cite{heltai2023reduced}.

When the ratio between the length and radius of the cylinder is large, it becomes attractive to approximate the pressure and flux inside the cylinder by using subspaces of functions that are constant along the cross-section. Such techniques lead to the simplified systems that are referred to as mixed-dimensional, which are analyzed in 
\cite{koch2020modeling,koppl2018mathematical,kuchta2021analysis, gjerde2020singularity,doi:10.1137/20M1362541}.

\bibliographystyle{siamplain}
\bibliography{references}
\end{document}


\maketitle

\section{A detailed example}

Here we include some equations and theorem-like environments to show
how these are labeled in a supplement and can be referenced from the
main text.
Consider the following equation:
\begin{equation}
  \label{eq:suppa}
  a^2 + b^2 = c^2.
\end{equation}
You can also reference equations such as \cref{eq:matrices,eq:bb} 
from the main article in this supplement.

\lipsum[100-101]

\begin{theorem}
  An example theorem.
\end{theorem}

\lipsum[102]
 
\begin{lemma}
  An example lemma.
\end{lemma}

\lipsum[103-105]

Here is an example citation: \cite{KoMa14}.

\section[Proof of Thm]{Proof of \cref{thm:bigthm}}
\label{sec:proof}
\lipsum[106-112]

\section{Additional experimental results}
\Cref{tab:foo} shows additional
supporting evidence. 

\begin{table}[htbp]
{\footnotesize
  \caption{Example table}  \label{tab:foo}
\begin{center}
  \begin{tabular}{|c|c|c|} \hline
   Species & \bf Mean & \bf Std.~Dev. \\ \hline
    1 & 3.4 & 1.2 \\
    2 & 5.4 & 0.6 \\ \hline
  \end{tabular}
\end{center}
}
\end{table}

\bibliographystyle{siamplain}
\bibliography{references}